\documentclass[12pt, a4paper]{article}

\usepackage{amsmath, amsfonts, amsthm}
\usepackage{graphicx}
\usepackage[english]{babel}

\newcommand{\R}{\mathbb{R}}

\newcommand{\transfer}{\mathcal{L}}

\DeclareMathOperator{\interior}{int}

\newtheorem{theorem}{Theorem}
\newtheorem{corollary}{Corollary}
\theoremstyle{definition}
\newtheorem{remark}{Remark}

\newtheorem{prop}{Proposition}
\newtheorem{lemma}{Lemma}

\begin{document}

\title{Smooth Liv\v{s}ic regularity for piecewise expanding maps}
\author{Matthew Nicol \and Tomas Persson\thanks{M.\ N.\ and
    T.\ P.\ would like to thank the support and hospitality of
    Institute Mittag-Leffler, Djursholm, Sweden. The authors would
    also like to thank Viviane Baladi for providing some comments and
    references. Tomas Persson was supported by EC FP6 Marie Curie ToK
    programme CODY.}}  \maketitle

\begin{abstract}
We consider the regularity of measurable solutions $\chi$ to the
cohomological equation
\[
  \phi = \chi \circ T -\chi,
\]
where $(T,X,\mu)$ is a dynamical system and $\phi \colon X\rightarrow
\R$ is a $C^k$ valued cocycle in the setting in which $T \colon
X\rightarrow X$ is a piecewise $C^k$ Gibbs--Markov map, an affine
$\beta$-transformation of the unit interval or more generally a
piecewise $C^{k}$ uniformly expanding map of an interval.  We show
that under mild assumptions, bounded solutions $\chi$ possess $C^k$
versions. In particular we show that if $(T,X,\mu)$ is a
$\beta$-transformation then $\chi$ has a $C^k$ version, thus improving a
result of Pollicott et al.~\cite{Pollicott-Yuri}.
\end{abstract}


\section{Introduction} 
In this note we consider the regularity of solutions $\chi$ to the
cohomological equation
\begin{equation} \label{coho}
  \phi = \chi \circ T -\chi
\end{equation}
where $(T,X,\mu)$ is a dynamical system and $\phi \colon X\rightarrow
\R$ is a $C^k$ valued co\-cycle. In particular we are interested in the
setting in which $T \colon X\rightarrow X$ is a piecewise $C^k$
Gibbs--Markov map, an affine $\beta$-transformation of the unit
interval or more generally a piecewise $C^{k}$ uniformly expanding map
of an interval. Rigidity in this context means that a
solution $\chi$ with a certain degree of regularity is forced by the
dynamics to have a higher degree of regularity.  Cohomological equations arise
frequently in ergodic theory and dynamics and, for example, determine
whether observations $\phi$ have positive variance in the central
limit theorem and and have implication for other distributional limits
(for examples see~\cite{ParryPollicottBook,ADSZ}). Related
cohomological equations to Equation~\eqref{coho} decide on stable
ergodicity and weak-mixing of compact group extensions of hyperbolic
systems~\cite{Keynes-Newton,ParryPollicottBook,noorani} and also play
a role in determining whether two dynamical systems are (H\"older,
smoothly) conjugate to each other.

Liv\v{sic}~\cite{Livsic1,Livsic2} gave seminal results on the 
regularity of measurable solutions to cohomological equations for
Abelian group extensions of Anosov systems with an absolutely
continuous invariant measure.  Theorems which establish that a priori
measurable solutions to cohomological equations must have a higher
degree of regularity are often called measurable Liv\v{s}ic theorems
in honor of his work.

We say that $\chi \colon X \to \R$ has a $C^k$ version (with respect
to $\mu$) if there exists a $C^k$ function $h \colon X\to \R$ such
that $h(x) =\chi (x)$ for $\mu$ \mbox{a.e.} $x\in X$.
 
Pollicott and Yuri~\cite{Pollicott-Yuri} prove Liv\v{s}ic theorems for
H\"{o}lder $\R$-extensions of $\beta$-transformations ($T \colon
[0,1)\rightarrow [0,1)$, $T(x)=\beta x \pmod 1$ where $\beta>1$) via
transfer operator techniques. They show that any essentially bounded
measurable solution $\chi$ to Equation~\eqref{coho} is of bounded
variation on $[0,1-\epsilon)$ for any $\epsilon>0$.  In this paper we
improve this result to show that measurable coboundaries $\chi$ for
$C^k$ $\R$-valued cocycles $\phi$ over $\beta$-transformations have $C^k$
versions (see~Theorem~\ref{the:smoothlivsic}).

Jenkinson~\cite{Jenkinson} proves that integrable measurable
coboundaries $\chi$ for $\R$-valued smooth cocycles $\phi$
(\mbox{i.e.} again solutions to $\phi=\chi \circ T-\chi$) over smooth
expanding Markov maps $T$ of $S^1$ have versions which are smooth on
each partition element.
% We show in
%Theorem~\ref{thm_markov} that measurable coboundaries $\chi$ for $C^2$
%$\R$-valued cocycles over Gibbs--Markov maps have versions which are
%$C^2$ on each element $T\alpha$ for each partition element $\alpha \in
%\mathcal{P}$ ($\mathcal{P}$ the defining Gibbs--Markov partition).

Nicol and Scott~\cite{NicolScott} have obtained measurable Liv\v{s}ic
theorems for certain discontinuous hyperbolic systems, including
$\beta$-transformations, Markov maps, mixing Lasota--Yorke maps, a
simple class of toral-linked twist map and Sinai dispersing
billiards. They show that a measurable solution $\chi$ to
Equation~\eqref{coho} has a Lipschitz version for
$\beta$-transformations and a simple class of toral-linked twist
map. For mixing Lasota--Yorke maps and Sinai dispersing billiards they
show that such a $\chi$ is Lipschitz on an open set. There is an
error in~\cite[Theorem 1]{NicolScott} in the setting of $C^2$ Markov
maps --- they only prove measurable solutions $\chi$ to
Equation~\eqref{coho} are Lipschitz on each element $T\alpha$, $\alpha
\in \mathcal{P}$, where $\mathcal{P}$ is the defining partition for
the Markov map, and not that the solutions are Lipschitz on $X$, as
Theorem 1 erroneously states. The error arose in the following way: if
$\chi$ is Lipschitz on $\alpha\in \mathcal{P}$ it is possible to
extend $\chi$ as a Lipschitz function to $T\alpha$ by defining
$\chi(Tx)=\phi(x)+\chi (x)$, however extending $\chi$ as a Lipschitz
function from $\alpha$ to $T^2 \alpha$ via the relation $\chi (T^2 x)=
\phi(T x)+ \chi(T x)$ may not be possible, as $\phi \circ T$ may have
discontinuities on $T\alpha$.  In this paper we give an example,
(see~Section~\ref{eg}), which shows that for Markov maps this result cannot
be improved on.
 
Gou\"ezel \cite{Gouezel} has obtained similar results to Nicol and
Scott~\cite{NicolScott} for cocycles into Abelian groups over
one-dimensional Gibbs--Markov systems. In the setting of Gibbs--Markov
system with countable partition he proves any measurable solution
$\chi$ to Equation~\eqref{coho} is Lipschitz on each element
$T\alpha$, $\alpha \in \mathcal{P}$, where $\mathcal{P}$ is the
defining partition for the Gibbs--Markov map.

%showing that coboundary solutions
% taking values in Lie groups satisfying a pinching condition 
%(to ensure that the system is partially hyperbolic) 
%are Lipschitz if the cocycle is Lipschitz.
% The same techniques show that, for such systems,  measurable transfer functions 
%taking values in compact matrix groups have Lipschitz versions. 
%These results were applied  to  prove stable ergodicity for
%semisimple and Abelian compact group extensions of certain uniformly hyperbolic systems
%with singularities, including the $\beta$-transformation, Markov maps
%and mixing Lasota--Yorke maps.

In related work, Aaronson and Denker~\cite[Corollary 2.3]{AD} have
shown that if $(T,X,\mu,\mathcal{P})$ is a mixing Gibbs--Markov map
with countable Markov partition
$\mathcal{P}$ preserving a probability measure $\mu$ and $\phi \colon X\to \R^d$ is Lipschitz (with respect
to a metric $\rho$ on $X$ derived from the symbolic dynamics) then any
measurable solution $\chi \colon X\to \R^d$ to $\phi=\chi \circ T
-\chi$ has a version $\tilde{\chi}$ which is Lipschitz continuous,
\mbox{i.e.} there exists $C>0$ such that $d(\tilde \chi(x), \tilde
\chi(y)) \leq C \rho(x,y)$ for all $x,y \in T(\alpha)$ and each
$\alpha\in \mathcal{P}$.

Bruin et al.~\cite{BHN} prove measurable Liv\v{s}ic theorems for
dynamical systems modelled by Young towers and Hofbauer towers. Their
regularity results apply to solutions of cohomological equations posed
on H\'enon-like mappings and a wide variety of non-uniformly hyperbolic
systems. We note that Corollary~1 of~\cite[Theorem 1]{BHN} is not
correct --- the solution is H\"{o}lder only on $M_k$ and $TM_k$ rather
than $T^j M_k$ for $j>1$ as stated for reasons similar to those given
above for the result in Nicol et al.~\cite{NicolScott}. 
%The work of \cite{ADSZ} is a study of the statistical properties of fibred
%systems and gives   rigidity results which provide checkable conditions for
%the aperiodicity of cocycles (i.e. nonexistence
%of solutions $\psi$) which allow one to establish, for example,
% de Moivre's approximation for various systems, including the
% $\beta$-transformation. A related  result is given in ~\cite[Lemma 6.1.2]{Gou}. 

\section{Main results}

We first describe one-dimensional Gibbs--Markov maps.  Let $I \subset \R$ be a bounded
interval, and $\mathcal{P}$ a countable partition of $I$ into
intervals. We let $m$ denote Lebesgue measure. Let $T \colon I \to I$
be a piecewise $C^k$, $k\ge 2$, expanding map such that $T$ is $C^k$
on the interior of each element of $\mathcal{P}$ with $\lvert T'
\rvert > \lambda > 1$, and for each $\alpha \in \mathcal{P}$,
$T\alpha$ is a union of elements in $\mathcal{P}$. Let $P_n:=
\bigvee_{j=0}^{n} T^{-j}\mathcal{P}$ and $J_T:=\frac{d(m\circ
  T)}{dm}$.  We assume:

\begin{itemize}
\item[(i)] (Big images property) There exists $C_1>0$ such that
  $m(T\alpha)>C_1$ for all $\alpha \in \mathcal{P}$.

\item[(ii)] There exists $0< \gamma_1<1$ such that $m(\beta) <
  \gamma_1^n$ for all $\beta \in P_n$.

\item[(iii)] (Bounded distortion) There exists $0<\gamma_2<1$ and
  $C_2>0$ such that $|1 - \frac{J_T (x)}{J_T (y)} |<C_2 \gamma_2^n$
  for all $x,y\in \beta$ if $\beta\in P_n$.
\end{itemize}

Under these assumptions $T$ has an invariant absolutely continuous
probability measure $\mu$ and the density of $\mu$,
$h=\frac{d\mu}{dm}$ is bounded above and below by a constant $0<
C^{-1} \le h(x) \le C$ for $m$ \mbox{a.e.} $x\in I$.

Note that a Markov map satisfies (i), (ii) and (iii) for finite
partition $\mathcal{P}$.

It is proved in \cite{NicolScott} for the Markov case (finite
$\mathcal{P}$), and in \cite{Gouezel} for the Gibbs--Markov case
(countable $\mathcal{P}$) that if $\phi \colon I \to \R$ is H\"older
continuous or Lipschitz continuous, and $\phi = \chi \circ T - \chi$
for some measurable function $\chi \colon I \to \R$, then there exists
a function $\chi_0 \colon I \to \R$ that is H\"older or Lipschitz on
each of the elements of $\mathcal{P}$ respectively, and $\chi_0 =
\chi$ holds $\mu$ (or $m$) a.e.  A related result to~\cite{Gouezel} is
given in~\cite[Theorem 7]{BHN} where $T$ is the base map of a Young
Tower, which has a Gibbs--Markov structure. 

Fried~\cite{Fried} has shown that the transfer operator of a graph
directed Markov system with $C^{k, \alpha}$-contractions, acting on a
space of $C^{k,\alpha}$-functions, has a spectral gap. If we apply his
result to our setting, letting the contractions be the inverse
branches of a Gibbs--Markov map we can conclude that the transfer
operator of a Gibbs--Markov map acting on $C^k$-functions has a
spectral gap. As in Jenkinson's paper~\cite{Jenkinson} and with the
same proof, this gives us immediately the following proposition, which
is implied by the results of Fried and Jenkinson:
\begin{prop}\label{thm_markov}
  Let $T \colon T\to I$ be a mixing Gibbs--Markov map such that $T$ is $C^k$
  on each partition element and $T^{-1} \colon T(\alpha)\to \alpha$ is
  $C^k$ on each partition element $\alpha\in \mathcal{P}$.  Let $\phi
  \colon I \to \R$ be uniformly $C^k$ on each of the partition
  elements $\alpha \in \mathcal{P}$.  Suppose $\chi \colon I \to \R$
  is a measurable function such that $\phi = \chi \circ T -
  \chi$. Then there exists a function $\chi_0 \colon I \to \R$ such
  that $\chi_0$ is uniformly $C^{k}$ on $T\alpha$ for each partition
  element of $\alpha \in \mathcal{P}$, and $\chi_0 = \chi$ almost
  everywhere.
\end{prop}

\section{A counterexample}\label{eg}

We remark that in general, if $\phi = \chi \circ T - \chi$, one cannot
expect $\chi$ to be continuous on $I$ if $\phi$ is $C^k$ on $I$. We
give an example of a Markov map $T$ with Markov partition
$\mathcal{P}$, a function $\phi$ that is $C^k$ on $I$, and a function
$\chi$ that is $C^k$ on each element $\alpha$ of $\mathcal{P}$ such
that $\phi = \chi \circ T - \chi$, yet $\chi$ has no version that is
continuous on $I$.

Let $0 < c < \frac{1}{4}$. Put $d = 2 - 4 c$. Define $T \colon [0,1]
\to [0,1]$ by
\[
  T (x) = \left\{ \begin{array}{ll} 2 x + \frac{1}{2} & \text{if } 0
    \leq x \leq \frac{1}{4} \rule{0pt}{14pt} \\
        d (x -\frac{1}{2}) + \frac{1}{2} & \text{if } \frac{1}{4} < x
        < \frac{3}{4} \rule{0pt}{14pt} \\
        2 x - \frac{3}{2} & \text{if } \frac{3}{4} \leq x \leq 1
        \rule{0pt}{14pt}
        \end{array} \right..
\]
If $c = \frac{1}{8}$, then the partition
\begin{align*}
  \mathcal{P} = \biggl\{ & \Bigl[0, \frac{1}{8} \Bigr], \Bigl[
    \frac{1}{8}, \frac{1}{4} \Bigr], \Bigl[ \frac{1}{4}, \frac{1}{2} -
    \frac{1}{4d} \Bigr], \Bigl[ \frac{1}{2} - \frac{1}{4d},
    \frac{1}{2} \Bigr],\\ & \Bigl[ \frac{1}{2}, \frac{1}{2} +
    \frac{1}{4d} \Bigr], \Bigl[ \frac{1}{2} + \frac{1}{4d},
    \frac{3}{4} \Bigr], \Bigl[ \frac{3}{4}, \frac{7}{8} \Bigr], \Bigl[
    \frac{7}{8}, 1 \Bigr] \biggr\}
\end{align*}
is a Markov partition for $T$. Define $\chi$ such that $\chi$ is $0$
on $[ \frac{1}{2} - \frac{1}{4d}, \frac{1}{2}]$ and $1$ on $[
  \frac{1}{2}, \frac{1}{2} + \frac{1}{4d} ]$. On $[0, \frac{1}{4})$ we
  define $\chi$ so that $\chi (0) = 1$ and $\lim_{x \to \frac{1}{4}}
  \chi (x) = 0$, and on $(\frac{3}{4}, 1]$ we define $\chi$ so that
$\chi (1) = 0$ and $\lim_{x \to \frac{3}{4}} \chi (x) = 1$. For any
natural number $k$, this can be done so that $\chi$ is $C^k$ except at
the point $\frac{1}{2}$ where it has a jump. One easily check that
$\phi$ defined by $\phi = \chi \circ T - \chi$ is $C^k$. This is
illustrated in Figures~\ref{fig:graph}--\ref{fig:phi}.

\setlength{\unitlength}{0.3\textwidth}

\begin{figure}
  \begin{minipage}{0.48\textwidth}
    \begin{center}
      \includegraphics[width=\unitlength]{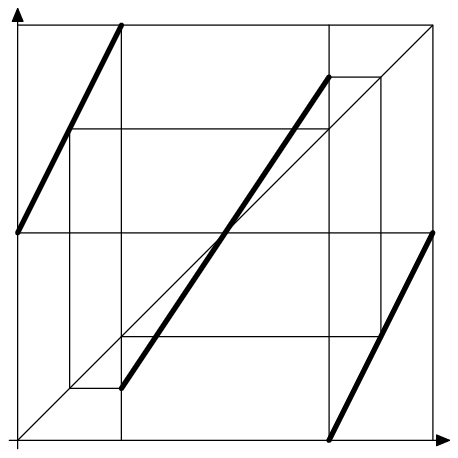}
      \caption{The graph of $T$.}
      \label{fig:graph}
      \bigskip
      
      \includegraphics[width=\unitlength]{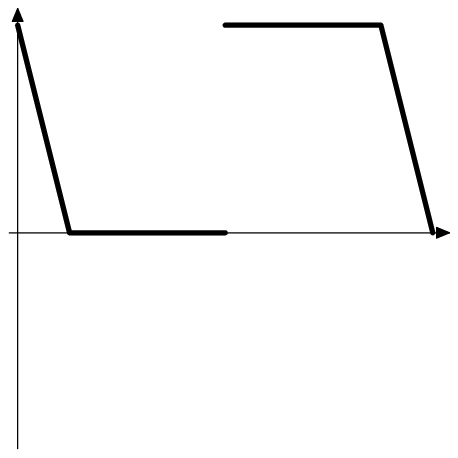}
      \caption{The graph of $\chi$.}
      \label{fig:chi}
    \end{center}
  \end{minipage}
  \hfill
  \begin{minipage}{0.48\textwidth}
    \begin{center}
      \includegraphics[width=\unitlength]{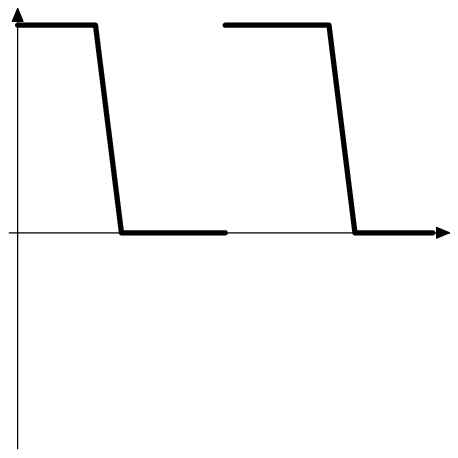}
      \caption{The graph of $\chi \circ T$.}
      \label{fig:chiT}
      \bigskip

      \includegraphics[width=\unitlength]{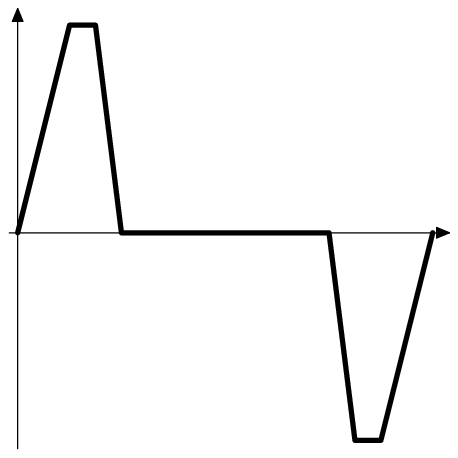}
      \caption{The graph of $\phi = \chi \circ T - \chi$.}
      \label{fig:phi}
    \end{center}
  \end{minipage}
\end{figure}

\section{Liv\v{s}ic  theorems for piecewise expanding\\ maps of an interval}
  
Let $I = [0, 1)$ and let $m$ denote Lebesgue measure on $I$. We consider piecewise expanding maps $T \colon I \to
  I$, satisfying the following assumptions: 

(i) There is a number $\lambda
  > 1$, and a finite partition $\mathcal{P}$ of $I$ into intervals,
  such that the restriction of $T$ to any interval in $\mathcal{P}$
  can be extended to a $C^2$-function on the closure, and $|T'| >
  \lambda$ on this interval. 

(ii)  $T$ has an absolutely
  continuous invariant measure $\mu$ with respect to which  $T$ is mixing.

(iii)  $T$ has the property of being weakly
  covering, as defined by Liverani in \cite{Liverani1}, namely that
  there exists an $n_0$ such that for any element $\alpha \in \mathcal{P}$
  \[
    \bigcup_{j = 0}^{n_0} T^j (\alpha) = I.
  \]

For any $n \geq 0$ we define the partition $\mathcal{P}_n =
\mathcal{P} \vee \cdots \vee T^{-n+1} \mathcal{P}$. The partition
elements of $\mathcal{P}_n$ are called $n$-cylinders, and
$\mathcal{P}_n$ is called the partition of $I$ into $n$-cylinders.

We prove the following two theorems.

\begin{theorem} \label{the:livsic}
 Let $(T,I,\mu)$ be a piecewise expanding map 
satisfying assumptions (i), (ii) and (iii).
 Let $\phi \colon I \to \R$ be a H\"older continuous function, such
  that $\phi = \chi \circ T - \chi$ for some measurable function $\chi$, with
  $e^{-\chi} \in L_1 (m)$. Then there exists a function $\chi_0$ such that
  $\chi_0$ has bounded variation and $\chi_0 = \chi$ almost
  everywhere.
\end{theorem}

For the next theorem we need some   more definitions. Let $A$ be a set,
and denote by $\interior A$ the interior of the set $A$. We assume
that the open sets $T (\interior \alpha)$, where $\alpha$ is an
element in $\mathcal{P}$, cover $\interior I$.

We will now define a new partition $\mathcal{Q}$. For a point $x$ in
the interior of some element of $\mathcal{P}$, we let $Q (x)$ be the
largest open set such that for any $x_2 \in Q (x)$, and any
$m$-cylinder $C_m$, there are points $(y_{1,k})_{k = 1}^n$ and
$(y_{2,k})_{k=1}^n$, such that $y_{1,k}$ and $y_{2,k}$ are in the same
element of $\mathcal{P}$, $T (y_{i, k+1}) = y_{i,k}$, $T (y_{1,1}) =
x$, $T (y_{2,1}) = x_2$, and $y_{1, n}, y_{2, n} \in C_m$. (This
forces $n \geq m$.)

Note that if $Q (x) \cap Q (y) \neq \emptyset$, then for $z \in Q (x)
\cap Q (y)$ we have $Q (z) = Q (x) \cup Q (y)$. We let $\mathcal{Q}$
be the coarsest collection of connected sets, such that any element of
$\mathcal{Q}$ can be represented as a union of sets $Q (x)$.

\begin{theorem} \label{the:smoothlivsic}
Let $(T,I,\mu)$ be a piecewise expanding map satisfying assumptions
(i), (ii) and  (iii).  If $\phi \colon I \to \R$ is a continuously differentiable function,
  such that $\phi = \chi \circ T - \chi$ for some function $\chi$ with
  $e^{-\chi} \in L_1(m)$, then there exists a function $\chi_0$ such that
  $\chi_0$ is continuously differentiable on each element of
  $\mathcal{Q}$ and $\chi_0 = \chi$ almost everywhere. If $T'$ is
  constant on the elements of $\mathcal{P}$, then $\chi_0$ is
  piecewise $C^k$ on $\mathcal{Q}$ if $\phi$ is in $C^k$. If for each $r$,
  $\frac{1}{ (T^r)'}$ is in $C^k$ with derivatives up to
  order $k$ uniformly bounded, then $\chi_0$ is piecewise $C^k$ on
  $\mathcal{Q}$ if $\phi$ is in $C^k$.
\end{theorem}

It is not always clear how big the elements in the partition
$\mathcal{Q}$ are. The following lemma gives a lower bound on the
diameter of the elements in $\mathcal{Q}$.

\begin{lemma}
  Assume that the sets $\{\, T (\interior \alpha) : \alpha \in
  \mathcal{P} \,\}$ cover $(0,1)$. Let $\delta$ be the Lebesgue number
  of the cover. Then the diameter of $\mathcal{Q} (x)$ is at least
  $\delta / 2$ for all $x$.
\end{lemma}

\begin{proof}
  Let $C_m$ be a cylinder of generation $m$. We need to show that for
  some $n \geq m$ there are sequences $(y_{1,k})_{k=1}^n$ and
  $(y_{2,k})_{k=1}^n$ as in the definition of $\mathcal{Q}$ above.

  Take $n_0$ such that $\mu (T^{n_0} (C_m)) = 1$. Write $C_m$ as a
  finite union of cylinders of generation $n_0$, $C_m =
  \bigcup_i D_i$. Then $R := [0, 1] \setminus T^{n_0} ( \cup_i \interior
  D_i )$ consists of finitely many points. Let $\varepsilon$ be the
  smallest distance between two of these points.

  Let $I_\delta$ be an open interval of diameter $\delta$. Let $n_1$ be
  such that $\delta \lambda^{-n_1} < \varepsilon$. Consider the full
  pre-images of $I_\delta$ under $T^{n_1}$. By the definition of $\delta$,
  there is at least one such pre-image, and any such pre-image is of
  diameter less than $\varepsilon$. Hence any pre-image contains at
  most one point from $R$.

  If the pre-image does not contain any point of $R$, then $I_\delta$
  is contained in some element of $\mathcal{Q}$ and we are
  done. Assume that there is a point $z$ in $I_\delta$ corresponding
  to the point of $R$ in the pre-image of $I_\delta$. Assume that $z$
  is in the right half of $I_\delta$. The case when $z$ is in the left
  part is treated in a similar way. Take a new open interval
  $J_\delta$ of length $\delta$, such that the left half of $J_\delta$
  coincides with the right half of $I_\delta$.

  Arguing in the same way as for $I_\delta$, we find that a pre-image
  of $J_\delta$ contains at most one point of $R$. If there is no such
  point, or the corresponding point $z_J \in J_\delta$ is not equal to
  $z$, then $I_\delta \cup J_\delta$ is contained in an element in
  $\mathcal{Q}$ and we are done.

  It remains to consider the case $z = z_J$. Let $I_\delta = (a, b)$
  and $J_\delta = (c, d)$. Then the intervals $(a, z)$ and $(z, d)$
  are both of length at least $\delta / 2$, and both are contained in
  some element of $\mathcal{Q}$. This finishes the proof.
\end{proof}

\begin{corollary}
  If $\beta>1$ and $T \colon x \mapsto \beta x \pmod 1$ is a
  $\beta$-transformation then clearly $T$ is weakly covering and
  $\mathcal{Q} = \{ (0,1) \}$, so in this case
  Theorem~\ref{the:smoothlivsic} and Theorem 1 of~\cite{NicolScott}
  imply that $\chi_0$ is in $C^k$ if $\phi$ is in $C^k$.
\end{corollary}
  
\begin{remark} 
  If $T \colon x \mapsto \beta x + \alpha \pmod 1$ is an affine
  $\beta$-transformation, then $\mathcal{Q} = \{ (0,1) \}$, and hence
  if $e^{-\chi}$ is in $L_1 (m)$ then $\chi$ has a $C^k$ version.
\end{remark}

\section{Proof of Theorem~\ref{the:livsic}}

We continue to assume that $(T,I,\mu)$ is  a  piecewise expanding map
satisfying assumptions (i), (ii) and (iii).
For a function $\psi \colon I \to \R$ we define the weighted transfer
operator $\transfer_\psi$ by
\[
  \transfer_\psi f (x) = \sum_{T (y) = x} e^{\psi (y)}
  \frac{1}{|\mathrm{d}_y T|} f(y).
\]

The proof is based on the following two facts, that can be found in 
 Hofbauer and Keller's papers \cite{Hofbauer-Keller1,
  Hofbauer-Keller2}. The first fact is
\begin{equation} \label{eq:L0}
  \begin{minipage}[c]{0.8\textwidth}
    There is a function $h \geq 0$ of bounded variation such that if
    $f \in L^1$ with $f \geq 0$ and $f \neq 0$, then $\transfer_0^n f$
    converges to $h \int f \, \mathrm{d} m$ in $L^1$.
  \end{minipage}
\end{equation}
The second fact is
\begin{equation} \label{eq:Lphi}
  \begin{minipage}[c]{0.8\textwidth}
    Let $f \in L^1$ with $f \geq 0$ and $f \neq 0$ be fixed. There is
    a function $w \geq 0$ with bounded variation, a measure $\nu$, and
    a number $a > 0$, depending on $\phi$, such that
    \[
      a^n \transfer^n_\phi f \to w \int f \, \mathrm{d} \nu,
    \]
    in $L^1$.
  \end{minipage}
\end{equation}
For $f$ of bounded variation, these facts are proved as
follows. Theorem~1 of \cite{Hofbauer-Keller1} gives us the desired
spectral decomposition for the transfer operator acting of functions
of bounded variation. Proposition~3.6 of Baladi's book \cite{Baladi}
gives us that there is a unique maximal eigenvalue. This proves the
two facts for $f$ of bounded variation. The case of a general $f$ in
$L^1$ follows since such an $f$ can be approximated by functions of
bounded variation.

Using that $T$ is weakly covering, we can conclude by Lemma~4.2 in
\cite{Liverani1}, that $h > \gamma > 0$. The proof of this fact in
\cite{Liverani1} goes through also for $w$, and so we may also
conclude that $w > \gamma > 0$.

Let us now see how Theorem~\ref{the:livsic} follows from these facts.
The following argument is analogous to the argument used by Pollicott
and Yuri in \cite{Pollicott-Yuri} for $\beta$-expansions.  We first
observe that $\phi = \chi \circ f - \chi$ implies that
\begin{align*}
  \transfer_\phi^n 1 (x) &= \sum_{T^n (y) = x} e^{S_n \phi (y)}
  \frac{1}{| \mathrm{d}_y T^n |} = \sum_{T^n (y) = x} e^{\chi (T^n y)
    - \chi (y)} \frac{1}{| \mathrm{d}_y T^n |} \\ &= e^{\chi (x)}
  \sum_{T^n (y) = x} e^{- \chi (y)} \frac{1}{| \mathrm{d}_y T^n |} =
  e^{\chi (x)} \transfer_0^n e^{- \chi} (x).
\end{align*}

%{\bf Comment: I still do not follow the argument involving the
%  sequence  $a_n$ given below. I have ordered the two hofbauer-keller
%papers from interlibrary loan as I cannot find my old copies and
%perhaps after reading them it will be clearer. In any case this is my
%only query.}

Since $a^n \transfer_\phi^n 1 \to w$ and $e^{-\chi} \transfer_\phi^n 1
= \transfer_0^n e^{- \chi} \to h \int e^{-\chi} \, \mathrm{d} m$ we
have that $a^n \transfer_\phi^n 1$ converges to $w$ in $L^1$ and
$\transfer_\phi^n 1$ converges to $h e^\chi \int e^{-\chi} \,
\mathrm{d}m$ in $L^1$. By taking a subsequence, we can achieve that
the convergences are a.e.
Therefore, we must have $a = 1$ and
\[
  w (x) = e^{\chi (x)} h (x) \int e^{- \chi} \, \mathrm{d} m,
    \quad \text{a.e.}
\]
It follows that
\[
  \chi (x) =\log w (x) - \log \int e^{- \chi} \, \mathrm{d} m -
  \log h (x),
\]
almost everywhere. Since $h$ and $w$ are bounded away from zero, their
logarithms are of bounded variation. This proves the theorem.

\section{Proof of Theorem~\ref{the:smoothlivsic}}

We first note that it is sufficient to prove that $\chi_0$ is
continuously differentiable on elements of the form $Q (x)$.

Let $x$ and $y$ satisfy $T (y) = x$. Then by $\phi = \chi \circ T -
\chi$ we have $\chi (x) = \phi (y) + \chi (y)$.

Let $x_1$ be a point in an  element of $\mathcal{Q}$, and take
$x_2 \in Q (x_1)$. We choose pre-images $y_{1,j}$ and $y_{2,j}$ of
$x_1$ and $x_2$ such that $T (y_{i,1}) = x_i$ and $T (y_{i,j}) =
y_{i,j-1}$. We then have
\[
  \chi (x_1) - \chi (x_2) = \sum_{j = 1}^n \bigl( \phi (y_{1,j}) -
  \phi (y_{2,j}) \bigr) + \chi (y_{1,n}) - \chi (y_{2,n}).
\]
We would like to let $n \to \infty$ and conclude that $\chi (y_{1,n})
- \chi (y_{2,n}) \to 0$. By Theorem~\ref{the:livsic} we know that
$\chi$ has bounded variation. Assume for contradiction that no matter
how we choose $y_{1,j}$ and $y_{2,j}$ we cannot make $|\chi (y_{1,n})
- \chi (y_{2,n})|$ smaller than some $\varepsilon > 0$. Let $m$ be
large and consider the cylinders of generation $m$. For any such
cylinder $C_m$, we can choose $y_{1,j}$ and $y_{2,j}$ such that
$y_{1,n}$ and $y_{2,n}$ both are in $C_m$. Since $|\chi (y_{1,n}) -
\chi (y_{2,n})| \geq \varepsilon$, the variation of $\chi$ on $C_m$ is
at least $\varepsilon$. Summing over all cylinders of generation $m$,
we conclude that the variation of $\chi$ on $I$ is at least $N (m)
\varepsilon$. Since $m$ is arbitrary and $N (m) \to \infty$ as $m \to
\infty$, we get a contradiction to the fact that $\chi$ is of bounded
variation.

Hence we can make $|\chi (y_{1,n}) - \chi (y_{2,n})|$ smaller that any
$\varepsilon > 0$ by choosing $y_{1,j}$ and $y_{2,j}$ in an
appropriate way. We conclude that
\[
  \chi (x_1) - \chi (x_2) = \sum_{j = 1}^\infty \bigl( \phi (y_{1,j})
  - \phi (y_{2,j}) \bigr).
\]
If $x_1 \neq x_2$ then $y_{1,j} \neq y_{2,j}$ for all $j$, and we
have
\[
  \frac{\chi (x_1) - \chi (x_2)}{x_1 - x_2} = \sum_{j = 1}^\infty
  \frac{\phi (y_{1,j}) - \phi (y_{2,j})}{y_{1,j} - y_{2,j}}
  \frac{y_{1,j} - y_{2,j}}{x_1 - x_2}.
\]
Clearly, the limit of the right hand side exists as $x_2 \to x_1$,
and is 
\[
  \sum_{j = 1}^\infty \phi' (y_{1,j}) \frac{1}{ (T^j)' (y_{1,j})}.
\]
The series converges since $| (T^j)' | > \lambda^j$.  This shows
that $\chi' (x_1)$ exists and satisfies
\begin{equation} \label{eq:chi'formula2}
  \chi' (x_1) = \sum_{j = 1}^\infty \phi' (y_{1,j}) \frac{1}{ (T^j)'
    (y_{1,j})}.
\end{equation}

If $T'$ is constant on the elements of $\mathcal{P}$, then
\eqref{eq:chi'formula2} implies that $\chi$ is in $C^k$ provided that
$\phi$ is in $C^k$.

Let us now assume that $\frac{1}{(T^r)^{'}}$  is in
$C^k$ with derivatives up to order $k$ uniformly bounded in $r$. We proceed
by induction. Let $g_n = \frac{1}{(T^n)'}$. Assume that
\begin{equation} \label{eq:higherder}
  \chi^{(m)} (x) = \sum_{n=1}^\infty \psi_{n,m} (y_n) g_n (y_n),
\end{equation}
where $(\psi_{n,m})_{n=1}^\infty$ is in $C^{n-m}$ with derivatives up
to order $n-m$ uniformly bounded. Then
\[
  \chi^{(m + 1)} (x)
  = \sum_{n=1}^\infty \bigl( \psi_{n,m}' (y_n) g_n (y_n) + \psi_{n,m}
  (y_n) g_n' (y_n) \bigr) g_n (y_n)
  = \sum_{n=1}^\infty \psi_{n,m + 1} g_n (y_n).
\]
This proves that there are uniformly bounded functions $\psi_{n,m}$
such that \eqref{eq:higherder} holds for $1 \leq m \leq k$. The series
in \eqref{eq:higherder} converges uniformly since $g_n$ decays with
exponential speed. This proves that $\chi$ is in $C^k$.
\qed


\newpage

\begin{thebibliography}{00}

\bibitem{AD} J. Aaronson and M. Denker. {\em Local limit theorems for
  partial sums of stationary sequences generated by Gibbs--Markov
  maps}, Stochast. Dynam, 1 (2001), 193--237.

\bibitem{ADSZ} J. \ Aaronson, M. \ Denker, O. \ Sarig,
  R. \ Zweim\"uller, {\em Aperiodicity of cocycles and stochastic
    properties of non-Markov maps,} Stoch. Dyn. {\bf 4} (2004),
  31--62.

\bibitem{Baladi} V. Baladi, {\em Positive transfer operators and decay
  of correlations}, Advanced Series in Nonlinear Dynamics 16, World
  Scientific Publishing, River Edge, 2000, ISBN 981-02-3328-0.

\bibitem{BHN} H. Bruin, M. Holland and M. Nicol. {\em Liv\v{s}ic
  regularity for Markov systems}, Ergod. Th. and Dynam. Sys (2005),
  25, 1739--1765.

\bibitem{dLMM} R.\ de la Llave, J.\ M.\ Marco, E.\ Moriyon, {\em
  Canonical perturbation theory of Anosov systems and regularity
  results for the Livsic cohomological equation}, Annals of Math. {\bf
  123} (1986) 537--611.
  
\bibitem{Fried} D. Fried, {\em The flat-trace asymptotics of a uniform
  system of contractions}, Ergodic Theory Dynam. Systems 15 (1995),
  no. 6, 1061--1073.

\bibitem{Gouezel} S. Gou\"ezel, {\em Regularity of coboundaries for
  non uniformly expanding maps}, Proceedings of the American
  Mathematical Society 134:2 (2006), 391--401.

\bibitem{Hofbauer-Keller1} F.~Hofbauer, G.~Keller, {\em Ergodic
  Properties of Invariant Measures for Piecewise Monotonic
  Transformations}, Math. Z. 180, (1982), 119--140

\bibitem{Hofbauer-Keller2} F.~Hofbauer, G.~Keller, {\em Equilibrium
  states for piecewise monotonic transformations}, Ergodic Theory and
  Dynamical Systems. 2, (1982), 23--43

\bibitem{Jenkinson} O.\ Jenkinson.  {\em Smooth cocycle rigidity for
  expanding maps and an application to Mostow rigidity,}
  Math. Proc. Camb. Phil. Soc. {\bf 132} (2002) 439--452.

\bibitem{Keynes-Newton} H. \ Keynes, D. \ Newton.  {\em Ergodic
  measures for non-abelian compact group extensions.}  Compositio
  Math. {\bf 32} (1976) 53--70.

%\bibitem{Slang} S.\ Lang,
%{\em Real and functional analysis,}
%Springer-Verlag 1993.

\bibitem{Liverani1} C. Liverani, {\em Decay of Correlations for
  Piecewise Expanding Maps}, Journal of Statistical Physics 78, 1995,
  1111--1129.

%\bibitem{Liverani2} C. Liverani, {\em Decay of Correlations}, Annals
%  of Mathematics, 142 (1995), no. 2, 239--301.

\bibitem{Livsic1} A.  N.  \ Liv\v{s}ic.  {\em Cohomology of Dynamical
  Systems}, Mathematics of the USSR Izvestija, {\bf 6}(6), (1972),
  1278--1301.

\bibitem{Livsic2} A.  N. \ Liv\v{s}ic.  {\em Homology properties of
  Y-Systems}, Math. Notes, {\bf 10} (1971) 758--763.

\bibitem{NicolScott} M. Nicol, A. Scott, {\em Liv\v{s}ic theorems and
  stable ergodicity for group extensions of hyperbolic systems with
  discontinuities}, Ergodic Theory and Dynamical Systems 23 (2003),
  1867--1889.
   
\bibitem{NP1} M. \ Nicol, M. \ Pollicott, {\em Measurable cocycle
  rigidity for some noncompact groups}, {Bull. Lond. Math. Soc.} {\bf
  31}, (1999), 592--600.

\bibitem{NP2} M. \ Nicol, M. \ Pollicott, {\em Liv\v{s}ic theorems for
  semisimple Lie groups}, {Erg. Th. and Dyn. Sys.} {\bf 21}, (2001),
  1501--1509.

\bibitem{NS} M.\ Nicol, A.\ Scott, {\em Liv\v{s}ic theorem and stable
  ergodicity for group extensions of hyperbolic systems with
  discontinuities,} {Ergod. Th. and Dyn. Sys.} {\bf 23}, (2003),
  1867--1889.

\bibitem{noorani} M.\ Noorani, {\em Ergodicity and weak-mixing of
  homogeneous extensions of measure-preserving transformations with
  applications to Markov shifts,} Monatsh. Math. {\bf 123} (1997)
  149--170.

%    t. Math. {\bf 105} (1991) 123--136.

\bibitem{ParryPollicottBook} W. Parry and M. Pollicott, {\em Zeta
  functions and closed orbits for hyperbolic systems}, Asterisque
  (Soc. Math. France), 187--188 (1990) 1--268.

\bibitem{PP1} W. \ Parry, M. \ Pollicott.  {\em The Livsic cocycle
  equation for compact Lie group extensions of hyperbolic
  systems}. J. London Math.  Soc.  (2) {\bf 56 } (1997) 405--416.

\bibitem{Pollicott-Walkden} M.\ Pollicott, C.\ Walkden, {\em
  Liv\u{s}ic theorems for connected Lie groups,}
  Trans. Amer. Math. Soc. {\bf 353} (2001) 2879--2895.

\bibitem{Pollicott-Yuri} M.\ Pollicott, M.\ Yuri, {\em Regularity of
  solutions to the measurable Livsic equation,}
  Trans. Amer. Math. Soc. {\bf 351} (1999) 559--568.

\bibitem{Scott} A. \ Scott, {\em Liv\v{s}ic theorems for unimodal
  maps}, In preparation.

\bibitem{thesis} A. \ Scott, {\em Liv\v{s}ic theorems and the stable
  ergodicity of compact group extensions of systems with some
  hyperbolicity}, Thesis, University of Surrey (2003).

%\bibitem{Walters} P. Walters, I don't know which paper you want to
%  cite. This also follows by Theorem~1 in Chapter~10, paragraph~4 of
%  Cornfeld, Fomin, Sinai, {\em Ergodic Theory}, Grundlehren der
%  Mathematischen Wissenschaften 245, Springer Verlag, New York, 1982.

\bibitem{Walkden1} C.\ Walkden, {\em Livsic theorems for hyperbolic
  flows,} Trans. Amer. Math. Soc. {\bf 352} (2000), 1299--1313.

\bibitem{Walkden2} C.\ Walkden, {\em Livsic regularity theorems for
  twisted cocycle equations over hyperbolic systems,} J. London
  Math. Soc., {\bf 61}, (2000), 286--300.

\end{thebibliography}
\end{document}